\documentclass[10pt]{amsart}

\usepackage{amsmath,xspace,amssymb,mathrsfs}
\usepackage{color}

\input xy
\xyoption{all}
\xyoption{2cell}
\UseAllTwocells
\CompileMatrices

\newcommand{\Spec}{\operatorname{Spec}}
\renewcommand{\phi}{\varphi}
\newcommand{\rn}{\operatorname{rank}}
\newcommand{\Ker}{\operatorname{Ker}}

\newcommand{\Ima}{\operatorname{Im}}

\newcommand{\Max}{\operatorname{Max}}

\newcommand{\Min}{\operatorname{Min}}
\newcommand{\Ann}{\operatorname{Ann}}
\newcommand{\Hom}{\operatorname{Hom}}

\newcommand{\Supp}{\operatorname{Supp}}
\newcommand{\tr}{\operatorname{tr}}

\newcommand{\Cl}{\operatorname{\mathfrak{C}}}

\newcommand{\Pic}{\operatorname{Pic}}

\newtheorem{proposition}{Proposition}[section]
\newtheorem{lemma}[proposition]{Lemma}

\newtheorem{corollary}[proposition]{Corollary}
\newtheorem{theorem}[proposition]{Theorem}

\theoremstyle{definition}

\newtheorem{example}[proposition]{Example}

\newtheorem{remark}[proposition]{Remark}

\usepackage{etoolbox}
\makeatletter
\patchcmd{\@settitle}{\uppercasenonmath\@title}{}{}{}
\patchcmd{\@setauthors}{\MakeUppercase}{}{}{}
\makeatother

\begin{document}

\title[Grothendieck ring vs Picard group]{On the Grothendieck ring and the relation of its group of units with the Picard group}

\author[A. Tarizadeh]{Abolfazl Tarizadeh}
\address{Department of Mathematics, Faculty of Basic Sciences, University of Maragheh, Maragheh, East Azerbaijan Province, Iran.}
\email{ebulfez1978@gmail.com}

\date{}

\begin{abstract} As the first main result of this article, we prove that if $e$ and $e'$ are idempotents of a commutative ring $A$, then there is a canonical isomorphism of $A$-modules: $$Ae\oplus Ae'\simeq Ae/Ae(1-e')\oplus Ae'/Ae'(1-e)\oplus A(e+e'-2ee').$$
This result plays an important role in proving several results on the Grothendieck ring $K_{0}(A)$.  
Especially, we first show that for any ring $A$ there is a complex of Abelian groups which is exact at the beginning and end: $$\xymatrix{0\ar[r]&\Pic(A)\ar[r]&K_{0}(A)^{\ast}
\ar[r]&\mathscr{B}(A)\ar[r]&0.}$$ 
Then we show that the above sequence is split exact for some certain rings $A$ (including Dedekind domains or more generally Noetherian one dimensional rings). 
The next main result asserts that for any ring $A$ we have the canonical isomorphisms of Abelian groups $\mathscr{B}(A)\simeq\mathscr{B}\big(K_{0}(A)\big)\simeq H_{0}(A)^{\ast}$. As an application, we show that a morphism of rings $A\rightarrow B$ lifts idempotents if and only if the induced ring map $K_{0}(A)\rightarrow K_{0}(B)$ lifts idempotents. If moreover, $B$ has finitely many maximal ideals then the map $K_{0}(A)\rightarrow K_{0}(B)$ is surjective. 
Finally, we show that the support of a finitely generated projective module is the whole prime spectrum if and only if its trace ideal is the whole unit ideal. 
\end{abstract}

\maketitle

\section{Introduction}

In this article, for any commutative ring $A$, we are interested in studying the Grothendieck ring $K_{0}(A)$ which is constructed in the most standard and canonical way. In particular, we are interested in studying the group of units of $K_{0}(A)$.

To achieve this goal, we first obtain an interesting formula regarding idempotents. In fact, we show that if $e$ and $e'$ are idempotents of a commutative ring $A$, then we have the following  canonical isomorphism of $A$-modules: $$Ae\oplus Ae'\simeq Ae/Ae(1-e')\oplus Ae'/Ae'(1-e)\oplus A(e+e'-2ee').$$ 
It is worth noting that during our scientific correspondence with Pierre Deligne, he at first somewhat doubted the correctness of this formula, but in the end the author managed to prove it. 

Next, in Theorem \ref{Theorem Pic-iv}, we prove that for any commutative ring $A$ we have the following complex of Abelian groups which is exact at the beginning and end places: $$\xymatrix{0\ar[r]&\Pic(A)\ar[r]^{f}&K_{0}(A)^{\ast}
\ar[r]^{g}&\mathscr{B}(A)\ar[r]&0.}$$
In the literature there is no known general result to compute the group of units of the Grothendieck ring $K_{0}(A)$. In this article, we compute this group for a certain class of rings (including Dedekind domains and more generally one dimensional Noetherian rings). In fact, in Corollary \ref{coro lbp 24}, we show that for such a ring $A$, we have the canonical isomorphism of groups $K_{0}(A)^{\ast}\simeq\Pic(A)\oplus\mathscr{B}(A)$. 

Grothendieck discovered algebraic K-theory in the late 1950s during his proof of the Grothendieck–Riemann–Roch theorem. In particular, he proved a fundamental result which asserts that for any commutative ring $A$, then the ring $K_{0}(A)$ modulo its nil-radical is canonically isomorphic to $H_{0}(A)$, the ring of all continuous functions $\Spec(A)\rightarrow\mathbb{Z}$ (for the proof see \cite[Chap. IX, \S3, Proposition 4.6]{Bass} or  \cite[Corollary 10.7]{Swan} or \cite[Chap. II, \S4, Corollary 4.6.1]{Weibel}). 

In this article, we will use the whole strength of Grothendieck's theorem and the above result that we obtained on idempotents, to prove several results on the Grothendieck ring $K_{0}(A)$. In particular, we obtain the following canonical isomorphisms of groups  $\mathscr{B}(A)\simeq\mathscr{B}\big(K_{0}(A)\big)\simeq H_{0}(A)^{\ast}$ (see Lemma \ref{Lemma idempots iso gro} and Theorem \ref{Corollary idemps of gr}). As an application, we show that a morphism of rings $A\rightarrow B$ lifts idempotents if and only if the induced ring map $K_{0}(A)\rightarrow K_{0}(B)$ lifts idempotents. If moreover, $B$ has finitely many maximal ideals, then we show that the canonical map $K_{0}(A)\rightarrow K_{0}(B)$ is surjective (see Theorem \ref{coro 23 surj}). We also show that for any nonzero ring $A$, then $K_{0}(A)$ is always an infinite ring of characteristic zero. 

In Lemma \ref{lemma Pic-iii}, we show that the support of a finitely generated projective module is the whole prime spectrum if and only if its trace ideal is the whole unit ideal of the ring.

\section{Preliminaries}

In this section, we recall some basic background for the convenience of the reader. In this article, all mathematical objects (monoids, groups, semirings and rings) are commutative. The group of units (invertible elements) of a ring $A$ is denoted by $A^{\ast}$. The set of zero-divisors of $A$ is denoted by $Z(A)$. 

If $M$ is a finitely generated flat module over a ring $A$, then for each $\mathfrak{p}\in\Spec(A)$, there exists a (unique) natural number $n_{\mathfrak{p}}\geqslant0$ such that $M_{\mathfrak{p}}\simeq
(A_{\mathfrak{p}})^{n_{\mathfrak{p}}}$ as  $A_{\mathfrak{p}}$-modules, because it is well known that every finitely generated flat module over a local ring is a free module (see \cite[Theorem 7.10]{Matsumura}). In fact, this number $n_{\mathfrak{p}}$ is the dimension of $\kappa(\mathfrak{p})$-vector space $M\otimes_{A}\kappa(\mathfrak{p})$ where $\kappa(\mathfrak{p})=A_{\mathfrak{p}}/
\mathfrak{p}A_{\mathfrak{p}}$ is the residue field of $A$ at $\mathfrak{p}$.
Hence, we obtain a function  $\mathrm{r}_{M}:\Spec(A)\rightarrow\mathbb{Z}$ given by $\mathfrak{p}\mapsto
\rn_{R_{\mathfrak{p}}}(M_{\mathfrak{p}})=n_{\mathfrak{p}}$. This function is called the rank map of $M$. It is well known that the rank map of a finitely generated flat $A$-module is continuous if and only if it is a projective $A$-module.  

If $G=\{[a,b]: a,b\in M\}$ is the Grothendieck group of a commutative monoid $M$, then the canonical map $f:M\rightarrow G$ given by $m\mapsto[m,0]$ is a morphism of monoids and the pair $(G,f)$ satisfies in the following universal property: for any such pair $(H,g)$, i.e., $H$ is an Abelian group and $g:M\rightarrow H$ is a morphism of monoids, then there exists a unique morphism of groups $h:G\rightarrow H$ such that $g=hf$. In fact, $h([a,b])=g(a)-g(b)$. 

Let $S$ be a semiring. The Grothendieck group $G(S)$ of the additive monoid $(S,+)$ can be made into a ring by defining the multiplication on it as $[a,b]\cdot[c,d]=[ac+bd, ad+bc]$. The multiplicative identity of this ring is $[1,0]$. The ring $G(S)$ is called the \emph{Grothendieck ring} of the semiring $S$. The canonical map $f:S\rightarrow G(S)$ given by $s\mapsto[s,0]$ is a morphism of semirings and the pair $\big(G(S),f\big)$ satisfies in the following universal property: For each such pair $(A,g)$, i.e. $g:S\rightarrow A$ is a morphism of semirings into a ring $A$, then there exists a unique morphism of rings $h:G(S)\rightarrow A$ such that $g=hf$. 

Let $A$ be a ring. By $S(A)$ we mean the set of isomorphism classes of finitely generated projective $A$-modules which is a semiring with the following operations. If $M$ and $N$ are finitely generated projective $A$-modules, then the addition is defined as $[M]+[N]=[M\oplus N]$ and the multiplication is defined as $[M]\cdot[N]=[M\otimes_{A}N]$. The isomorphism class of the zero module is the additive identity of this semiring, and the isomorphism class of $A$ is the multiplicative identity of this semiring. In this semiring, we will often denote the isomorphism class $[M]$ simply by $M$ if there is no confusion. We denote the Grothendieck ring of the semiring $S(A)$ by $K_{0}(A)$. The ring $K_{0}(A)$ is of particular interest in mathematics, especially in algebraic K-theory. Sometimes by abuse of the terminology, $K_{0}(A)$ is also called the Grothendieck ring of $A$.
Every morphism of rings $\phi:A\rightarrow B$ induces a morphism of semirings $S(A)\rightarrow S(B)$ that is given by $M\mapsto M\otimes_{A}B$. Then by the universal property of Grothendieck rings, we obtain a (unique) morphism of rings $K_{0}(\phi):K_{0}(A)\rightarrow K_{0}(B)$ which is given by $[M,N]\mapsto[M\otimes_{A}B, N\otimes_{A}B]$. 
It can be easily seen that every element of $K_{0}(A)$ is also of the form $[P,A^{d}]$ where $P$ is a finitely generated projective $A$-module and $d\geqslant0$. Also note that in $K_{0}(A)$ we have $[M,A^{m}]=[N,A^{n}]$ if and only if $M\oplus A^{n+d}\simeq N\oplus A^{m+d}$ as $A$-modules for some $d\geqslant0$. 

For any ring $A$, by $H_{0}(A)$ we mean the ring of all continuous functions $\Spec(A)\rightarrow\mathbb{Z}$ where $\mathbb{Z}$ is equipped with the discrete topology. If $\phi:A\rightarrow B$ is a morphism of rings then the map $H_{0}(\phi):H_{0}(A)\rightarrow H_{0}(B)$ given by $f\mapsto f\phi^{\ast}$ is a morphism of rings where the map $\phi^{\ast}:\Spec(B)\rightarrow\Spec(A)$ is induced by $\phi$. For more information on this ring see e.g. \cite[\S5]{A. Tarizadeh Racsam2} or \cite{Weibel} or \cite{Swan} or \cite{Bass}. The map $S(A)\rightarrow H_{0}(A)$ given by $[M]\mapsto\mathrm{r}_{M}$ is a morphism of semirings where $\mathrm{r}_{M}$ denotes the rank map of $M$. Then by the universal property of Grothendieck rings, we obtain a (unique) morphism of rings $K_{0}(A)\rightarrow H_{0}(A)$ which is given by $[M,N]\mapsto\mathrm{r}_{M}-\mathrm{r}_{N}$.  

\section{Finitely generated projective modules and Picard group}

In this section, we first prove some results on finitely generated projective modules that slightly improve the related results in the literature. Then we derive the Picard group construction from these results. 

Recall that if $M$ is a module over a ring $A$ then we have a canonical morphism of $A$-modules $\widehat{M}\otimes_{A} M\rightarrow A$ that is given by $f\otimes x\mapsto f(x)$ where $\widehat{M}=\Hom_{A}(M,A)$ is the dual module of $M$. The image of this map is an ideal of $A$ which is called the \emph{trace ideal of $M$} and is denoted by $\tr_{A}(M)$ or simply by $\tr(M)$ if there is no confusion on the base ring $A$. 
The following result improves \cite[Chap 3, Prop. 20]{Silvester}.

\begin{lemma}\label{lemma Pic-iii} Let $M$ be a module over a ring $A$. If $\tr(M)=A$ then $\Supp(M)=\Spec(A)$. If moreover, $M$ is a finitely generated projective $A$-module the converse holds. 
\end{lemma}

\begin{proof} Suppose $M_{\mathfrak{p}}=0$ for some prime ideal $\mathfrak{p}$ of $A$. Since $\tr(M)=A$, we may write $1=\sum\limits_{i=1}^{n}f_{i}(x_{i})$ where $f_{i}\in\widehat{M}$ and $x_{i}\in M$ for all $i$. 
For each $i$, there is some $s_{i}\in A \setminus\mathfrak{p}$ such that $s_{i}x_{i}=0$. Take $s=s_{1}\ldots s_{n}$ then $s=\sum\limits_{i=1}^{n}sf_{i}(x_{i})=
\sum\limits_{i=1}^{n}f_{i}(sx_{i})=0$ which is a contradiction. Conversely, if $I=\tr(M)$ is a proper ideal of $A$, then $I\subseteq\mathfrak{p}$ for some $\mathfrak{p}\in\Spec(A)$. It is well known that if $M$ is a projective $A$-module, then $IM=M$. For its proof see e.g. \cite[Theorem 3.1]{A. Tarizadeh acta}. It follows that $(IR_{\mathfrak{p}})M_{\mathfrak{p}}=M_{\mathfrak{p}}$. Thus by the Nakayama lemma, $M_{\mathfrak{p}}=0$ which is a contradiction. 
\end{proof}

There is a minor gap in the last lines of the proof of \cite[Chap. III, Proposition 7.4]{Bass}. In the following result, we fill it in:

\begin{lemma}\label{Lemma Pic-ii H.Bass} Let $M$ and $N$ be modules over a ring $A$. If $M\otimes_{A}N\simeq A^{n}$ as $A$-modules for some $n\geqslant1$, then $M$ and $N$ are finitely generated projective $A$-modules.
\end{lemma}

\begin{proof} Clearly $M\otimes_{A}N=(x_{k}\otimes y_{k}: k=1,\ldots, d)$ is a finitely generated $A$-module where $d\geqslant n$. By the universal property of free modules, there is a (unique) morphism of $A$-modules $h:A^{d}\rightarrow M$ such that $h(\epsilon_{k})=x_{k}$ for all $k$. Since each $x_{k}\otimes y_{k}$ is in the image of the induced morphism $h\otimes1_{N}:A^{d}\otimes_{A}N\rightarrow M\otimes_{A}N$, thus it is surjective. In fact, $h\otimes1_{N}$ is a split epimorphism. That is, we have the following split exact sequence:  $$\xymatrix{0\ar[r]&K\ar[r]^{inc\:\:\:\:\:\:\:\:\:\:\:}&
A^{d}\otimes_{A}N
\ar[r]^{h\otimes1_{N}}&M\otimes_{A}N\ar[r]&0}$$ where $K$ is the kernel of $h\otimes1_{N}$. Note that split exact sequences are left split exact by additive functors. Hence, by applying the additive functor $M\otimes_{A}-$ to the above sequence, we obtain the following split exact sequence: $$\xymatrix{0\ar[r]&M\otimes_{R}K\ar[r]
&A^{nd}\ar[r]&M^{n}\ar[r]&0.}$$ Since $n\geqslant1$, so $M$ is a direct summand of the free $A$-module $A^{nd}$. Hence, $M$ is a finitely generated projective $A$-module. Similarly, $N$ is also a finitely generated projective $A$-module.
\end{proof}

Recall that if $M$ and $N$ are finitely
generated flat (resp. projective) modules over a ring $A$, then $M\oplus N$ and $M\otimes_{A}N$ are finitely generated flat (resp. projective) $A$-modules and we have $\mathrm{r}_{M\oplus N} = \mathrm{r}_{M}+\mathrm{r}_{N}$ and $\mathrm{r}_{M\otimes_{A}N} = \mathrm{r}_{M}\cdot\mathrm{r}_{N}$. 

In Lemma \ref{Lemma Pic-ii H.Bass}, if $n=1$ then $M$ and $N$ are finitely generated projective $A$-modules of constant rank 1. Also note that this result does not hold for $n=0$. For instance, let $I$ and $J$ be coprime ideals of a ring $A$, then $A/I\otimes_{A}A/J=0$ but $A/I$ and $A/J$ are not necessarily $A$-projective (nor $A$-flat). As a specific example, in the ring of integers $\mathbb{Z}$, take $I=2\mathbb{Z}$ and $J=3\mathbb{Z}$.

The following result shows that being of finite type in modules is a local property in the following sense: 

\begin{lemma}\label{Lemma 6 2026 f.g.} Let $M$ be a module over a ring $A$ with the property that for each $\mathfrak{p}\in\Spec(A)$ there exists some $f\in A\setminus\mathfrak{p}$ such that $M_{f}$ is a finitely generated $A_{f}$-module. Then $M$ is a finitely generated $A$-module.
\end{lemma}

\begin{proof} Using the quasi-compactness of $\Spec(A)$, there exist finitely many elements $f_{1},\ldots,f_{n}\in A$ such that $\Spec(A)=\bigcup\limits_{i=1}^{n}D(f_{i})$ and $M_{f_{i}}=(x_{i,1}/1,\ldots, x_{i,d_{i}}/1)$ is a finitely generated $A_{f_{i}}$-module for all $i\in\{1,\ldots,n\}$. We show that $M$ as $A$-module is generated by the elements $x_{i,1},\ldots, x_{i,d_{i}}$ with $i=1,\ldots,n$. If $m\in M$ then for each $i\in\{1,\ldots,n\}$ we may write $m/1=\sum\limits_{k=1}^{d_{i}}(r_{i,k}/f^{s_{k}}_{i})
(x_{i,k}/1)$. Thus there exists a natural number $N\geqslant1$ such that $f^{N}_{i}m=\sum\limits_{k=1}^{d_{i}}r'_{i,k}x_{i,k}$ for all $i\in\{1,\ldots,n\}$. We have $\Spec(A)=\bigcup\limits_{i=1}^{n}D(f^{N}_{i})$ and so
$1=\sum\limits_{i=1}^{n}r''_{i}f^{N}_{i}$. It follows that $m=\sum\limits_{i,k}r'_{i,k}r''_{i}x_{i,k}$. This completes the proof.
\end{proof}

\begin{lemma}\label{Lemma 7 2026 nice} Let $M$ be a module over a ring $A$ with the property that for each $\mathfrak{p}\in\Spec(A)$ there exists some $f\in A\setminus\mathfrak{p}$ such that $M_{f}$ is a free $A_{f}$-module of finite rank. Then $M$ is a projective $A$-module.
\end{lemma}

\begin{proof} By Lemma \ref{Lemma 6 2026 f.g.}, $M$ is a finitely generated $A$-module, and so there exists an exact sequence of $A$-modules: $$\xymatrix{0\ar[r]&K\ar[r]&A^{m}\ar[r]&M\ar[r]&0}$$ where $m\geqslant0$ is a natural number and $K=\Ker(A^{m}\rightarrow M)$. For each prime ideal $\mathfrak{p}$, choose some $f\in A\setminus\mathfrak{p}$ such that $M_{f}$ is a free over $A_{f}$. Then the following exact sequence: $$\xymatrix{0\ar[r]&K_{f}\ar[r]&(A_{f})^{m}
\ar[r]&M_{f}\ar[r]&0}$$ splits, because $M_{f}$ is free. Hence $K_{f}$ is a direct summand of $(A_{f})^{m}$ and thus a finitely generated $A_{f}$-module. Then, again by Lemma \ref{Lemma 6 2026 f.g.}, $K$ is a finitely generated $A$-module. This shows that $M$ is a finitely presented $A$-module. It is well known that every finitely presented flat module is a projective module. 
\end{proof}

\begin{remark}\label{Remark dort 2026} Recall that for any modules $M$ and $N$ over a ring $A$ and for any ring map $A\rightarrow B$ we have the following natural morphism of $B$-modules: $$\phi:\Hom_{A}(M, N)\otimes_{A}B\rightarrow
\Hom_{B}(M\otimes_{A}B, N\otimes_{A}B)$$ which is given by $f\otimes b\mapsto(f\otimes b)^{\ast}$ where the morphism of $B$-modules $(f\otimes b)^{\ast}:M\otimes_{A}B\rightarrow N\otimes_{A}B$ is defined as $x\otimes b'\mapsto f(x)\otimes bb'$. 

In addition, if $\mathfrak{p}$ is the contraction of a prime ideal $\mathfrak{q}$ of $B$ under the above ring  map, then the following diagram is commutative: $$\xymatrix{
(\Hom_{A}(M, N)\otimes_{A}B)_{\mathfrak{q}}
\ar[r]^{\phi_{\mathfrak{q}}\:\:\:\:\:\:\:\:}\ar[d]^{} &
(\Hom_{B}(M\otimes_{A}B, N\otimes_{A}B))_{\mathfrak{q}}\ar[d]^{}\\
\Hom_{A_{\mathfrak{p}}}(M_{\mathfrak{p}}, N_{\mathfrak{p}})\otimes_{A_{\mathfrak{p}}}B_{\mathfrak{q}}
\ar[r]^{}&
\Hom_{B_{\mathfrak{q}}}((M\otimes_{A}B)_{\mathfrak{q}}, (N\otimes_{A}B)_{\mathfrak{q}})}$$
where the unnamed arrows are the natural maps. 
\end{remark}

We have then the following result which generalizes \cite[Proposition 7.3]{Bass} (see Corollary \ref{Remark 1 2026}).

\begin{lemma}\label{Remark 2 2026} Let $M$ be a finitely presented module over a ring $A$ and $A\rightarrow B$ be a flat ring map. Then for any $A$-module $N$ the natural morphism of $B$-modules: $$\Hom_{A}(M, N)\otimes_{A}B\rightarrow
\Hom_{B}(M\otimes_{A}B, N\otimes_{A}B)$$ is an isomorphism.
\end{lemma}

\begin{proof} By applying the left exact contravariant functor $\Hom_{A}(-,N)$ to a finite presentation of $M$: $$\xymatrix{A^{n}\ar[r]&A^{m}\ar[r]&M\ar[r]&0}$$
we get the following exact sequence of $A$-modules: 
$$\xymatrix{0\ar[r]&\Hom_{A}(M,N)\ar[r]&
N^{m}\ar[r]&N^{n}.}$$
Since $B$ is $A$-flat, the following exact sequence of $B$-modules is obtained: $$\xymatrix{0\ar[r]&\Hom_{A}(M,N)\otimes_{A}B\ar[r]&
N^{m}\otimes_{A}B\ar[r]&N^{n}\otimes_{A}B.}$$
By applying the right exact functor $-\otimes_{A}B$ to the above finite presentation of $M$, we also get the following exact sequence of $B$-modules: $$\xymatrix{B^{n}\ar[r]&B^{m}
\ar[r]&M\otimes_{A}B\ar[r]&0.}$$
Then by applying the left exact contravariant functor $\Hom_{B}(-,N\otimes_{A}B)$ to the above exact sequence we get the following exact sequence of $B$-modules:
$$\xymatrix{0\ar[r]&\Hom_{B}(M\otimes_{A}B,N\otimes_{A}B)
\ar[r]&(N\otimes_{A}B)^{m}\ar[r]&
(N\otimes_{A}B)^{n}.}$$
It can be seen that the following diagram is commutative (with exact rows):
$$\xymatrix{0\ar[r]&\Hom_{A}(M,N)\otimes_{A}B
\ar[r]\ar[d]^{\phi}&
N^{m}\otimes_{A}B\ar[r]\ar[d]^{\simeq}&
N^{n}\otimes_{A}B\ar[d]^{\simeq} \\ 0
\ar[r]&\Hom_{B}(M\otimes_{A}B,N\otimes_{A}B)\ar[r]&
(N\otimes_{A}B)^{m}\ar[r]&(N\otimes_{A}B)^{n}.}$$ 
Then by the five lemma, the natural map $\phi$ is an isomorphism. 
\end{proof}

By invertible module over a ring $A$ we mean a finitely generated projective $A$-module of constant rank 1.

If $f:A\rightarrow B$ is a ring map and $M$ is a finitely
generated flat (resp. projective) module over $A$, then $M\otimes_{A}B$ is a finitely generated flat (resp. projective) module over $B$, and $\mathrm{r}_{M\otimes_{A}B}=\mathrm{r}_{M}\circ f^{\ast}$ where the map $f^{\ast}:\Spec(B)\rightarrow\Spec(A)$ is induced by $f$. In particular, if $M$ is an invertible $A$-module, then $M\otimes_{A}B$ is an invertible $B$-module.  

The above results provide an alternative proof to the following important fact:

\begin{theorem}\label{Coro toofan nice 1} For a module $M$ over a ring $A$ the following statements are equivalent: \\
$\mathbf{(i)}$ $M$ is an invertible $A$-module. \\
$\mathbf{(ii)}$ The natural morphism of $A$-modules $\widehat{M}\otimes_{A}M\rightarrow A$ 
is an isomorphism. \\
$\mathbf{(iii)}$ There are finitely many elements $f_{1},\ldots,f_{n}\in A$ which generate the unite ideal of $A$ and $M_{f_{i}}\simeq A_{f_{i}}$ as $A_{f_{i}}$-modules for all $i$. 
\end{theorem}

\begin{proof} (i)$\Rightarrow$(ii): Since $M$ is a finitely generated projective $A$-module, we have $M\oplus N\simeq A^{n}$ for some $A$-module $N$ and some $n\geqslant0$. It follows that $\Hom_{A}(M\oplus N, A)\simeq \widehat{M}\oplus\widehat{N}\simeq A^{n}$. This shows that $\widehat{M}$ is also a finitely generated projective $A$-module. In addition, the rank maps of $M$ and $\widehat{M}$ are the same, because by Lemma  \ref{Remark 2 2026} (note that every finitely generated projective module is finitely presented), we have $(\widehat{M})_{\mathfrak{p}}\simeq\widehat
{(M_{\mathfrak{p}})}$ for all $\mathfrak{p}\in\Spec(A)$. Then $\widehat{M}\otimes_{A}M$ is also a finitely generated projective $A$-module and its rank map $\mathrm{r}_{\widehat{M}\otimes_{A}M}=
(\mathrm{r}_{\widehat{M}})\cdot(\mathrm{r}_{M})$ is the constant function 1. By Lemma \ref{lemma Pic-iii}, the natural morphism $\widehat{M}\otimes_{A}M\rightarrow A$ is surjective. But it can be easily seen that if $N$ is a finitely generated flat  $A$-module and $N'$ is a finitely generated projective $A$-module whose rank maps are the same, then every 
surjective morphism of $A$-modules $f:N\rightarrow N'$ is an isomorphism. Thus  the natural morphism $\widehat{M}\otimes_{A}M\rightarrow A$ 
is an isomorphism. \\
(ii)$\Rightarrow$(i): This follows from Lemma \ref{Lemma Pic-ii H.Bass}. \\
(ii)$\Rightarrow$(iii): We may write $1=\sum\limits_{i=1}^{n}\phi_{i}(x_{i})$ where $\phi_{i}\in\widehat{M}$ and $x_{i}\in M$ for all $i$. Setting $f_i:=\phi_{i}(x_{i})$ for all $i$. For each $k$, consider the induced morphism of $A_{f_{k}}$-modules $\psi_{k}:M_{f_{k}}\rightarrow A_{f_{k}}$ which is given by $x/f^{d}_{k}\mapsto\phi_{k}(x)/f^{d}_{k}$. This map is surjective, because $\psi_{k}(x_{k}/1)=f_{k}/1$, a unit in $A_{f_{k}}$. Since $M$ is an invertible module over $A$, so $M_{f_{k}}$ is an invertible module over $A_{f_{k}}$. Hence, $\psi_{k}$ is an isomorphism. In fact, for any  $\phi\in\widehat{M}$ and for any $x\in M$, the induced map $M_{f}\rightarrow A_{f}$ given by $m/f^{d}\mapsto\phi(m)/f^{d}$ is an isomorphism of $A_{f}$-modules where $f:=\phi(x)$. \\
(iii)$\Rightarrow$(i): By Lemmas \ref{Lemma 6 2026 f.g.} and \ref{Lemma 7 2026 nice}, $M$ is a finitely generated projective $A$-module. It is also clear that $M_{\mathfrak{p}}\simeq A_{\mathfrak{p}}$ as $A_{\mathfrak{p}}$-modules for all $\mathfrak{p}\in\Spec(A)$, indeed $f_{i}\notin\mathfrak{p}$ for some $i$ and so we have the natural ring maps decomposition $\xymatrix{A\ar[r]&A_{f_{i}}\ar[r]&A_{\mathfrak{p}}}$, then
$M_{\mathfrak{p}}\simeq M\otimes_{A}A_{\mathfrak{p}}\simeq
M\otimes_{A}(A_{f_{i}\otimes_{A_{f_{i}}}}A_{\mathfrak{p}})
\simeq M_{f_{i}}\otimes_{A_{f_{i}}}A_{\mathfrak{p}}\simeq
A_{f_{i}}\otimes_{A_{f_{i}}}A_{\mathfrak{p}}\simeq A_{\mathfrak{p}}$. Hence, $M$ is an invertible $A$-module.  
\end{proof}

For any commutative ring $A$, by $\Pic(A)$ we mean the collection of isomorphism classes of invertible modules over $A$. It can be seen that $\Pic(A)$ is indeed a ``set". Using  Theorem \ref{Coro toofan nice 1}, then 
this set $\Pic(A)$ by the  operation $[L_{1}]\cdot [L_{2}]=[L_{1}\otimes_{A}L_{2}]$ is an Abelian group whose identity element is the isomorphism class of $A$, and the inverse of each $[L]\in\Pic(A)$ is the isomorphism class of $\widehat{L}$. The group $\Pic(A)$ is called the \emph{Picard group of $A$}. In this group $\Pic(A)$, we will often denote the isomorphism class $[L]$ simply by $L$ if there is no confusion. 
If $f:A\rightarrow B$ is a morphism of rings, then the map $\Pic(f):\Pic(A)\rightarrow\Pic(B)$ given by $L\mapsto L\otimes_{A}B$ is a morphism of groups. In fact, the Picard group construction is a covariant functor from the category of commutative rings to the category of Abelian groups.

\begin{remark}\label{Remark 3 2026} Let $(M_{i})$ be a direct system of $A$-modules over a directed poset $(I,<)$ and $N$ be an $A$-module. Then for any ring map $A\rightarrow B$, the following diagram is commutative: $$\xymatrix{
\Hom_{A}(\lim\limits_{\overrightarrow{i\in I}}M_{i}, N)\otimes_{A}B
\ar[r]^{}\ar[d]^{\simeq} &
\Hom_{B}(\lim\limits_{\overrightarrow{i\in I}}(M_{i}\otimes_{A}B), N\otimes_{A}B)\ar[d]^{\simeq}\\
\lim\limits_{\overleftarrow{i\in I}}
(\Hom_{A}(M_{i}, N)\otimes_{A}B)
\ar[r]^{}&
\lim\limits_{\overleftarrow{i\in I}}
\Hom_{B}(M_{i}\otimes_{A}B, N\otimes_{A}B)}$$
where the vertical maps are natural isomorphisms. 
\end{remark}
 
Lemma \ref{Remark 2 2026} has also the following consequence (note that the arrow $A\rightarrow B$ is an arbitrary ring map, not necessarily flat): 

\begin{corollary}\label{Remark 1 2026} If $M$ is a finitely generated projective module over a ring $A$, then for any $A$-module $N$ and for any ring map $A\rightarrow B$ the natural morphism of $B$-modules: $$\phi:\Hom_{A}(M, N)\otimes_{A}B\rightarrow
\Hom_{B}(M\otimes_{A}B, N\otimes_{A}B)$$ is an isomorphism. 
\end{corollary}

\begin{proof} It suffices to show that for any prime (even maximal) ideal $\mathfrak{q}$ of $B$ the induced map $\phi_{\mathfrak{q}}:(\Hom_{A}(M, N)\otimes_{A}B)_{\mathfrak{q}}\rightarrow
(\Hom_{B}(M\otimes_{A}B, N\otimes_{A}B))_{\mathfrak{q}}$ is an isomorphism. If $\mathfrak{p}$ is the contraction of $\mathfrak{q}$ under the above ring map, then by Remark \ref{Remark dort 2026}, the following diagram is commutative:
$$\xymatrix{
(\Hom_{A}(M, N)\otimes_{A}B)_{\mathfrak{q}}
\ar[r]^{\phi_{\mathfrak{q}}\:\:\:\:\:\:\:\:}\ar[d]^{} &
(\Hom_{B}(M\otimes_{A}B, N\otimes_{A}B))_{\mathfrak{q}}\ar[d]^{}\\
\Hom_{A_{\mathfrak{p}}}(M_{\mathfrak{p}}, N_{\mathfrak{p}})\otimes_{A_{\mathfrak{p}}}B_{\mathfrak{q}}
\ar[r]^{\phi'\:\:\:\:\:\:\:\:\:}&
\Hom_{B_{\mathfrak{q}}}((M\otimes_{A}B)_{\mathfrak{q}}, (N\otimes_{A}B)_{\mathfrak{q}})}$$
Since $M$ is a finitely presented $A$-module, and hence $M\otimes_{A}B$ is a finitely presented $B$-module (every finitely generated projective module is finitely presented), by Lemma \ref{Remark 2 2026} the vertical maps are isomorphisms. Thus it suffices to show that the bottom map $\phi'$ is an isomorphism.  We know that every (finitely generated) projective module over a local ring is a free module. Then $M_{\mathfrak{p}}\simeq(A_{\mathfrak{p}})^{d}$ and so $(M\otimes_{A}B)_{\mathfrak{q}}
\simeq(B_{\mathfrak{q}})^{d}$ for some $d\geqslant0$. Then by applying a particular case of Remark \ref{Remark 3 2026} for finite direct sums, we get the following commutative diagram: 
$$\xymatrix{
\Hom_{A_{\mathfrak{p}}}(M_{\mathfrak{p}}, N_{\mathfrak{p}})\otimes_{A_{\mathfrak{p}}}B_{\mathfrak{q}}
\ar[r]^{\phi'\:\:\:\:\:\:\:\:\:}\ar[d]^{\simeq} &
\Hom_{B_{\mathfrak{q}}}((M\otimes_{A}B)_{\mathfrak{q}}, (N\otimes_{A}B)_{\mathfrak{q}})\ar[d]^{\simeq}\\
(N_{\mathfrak{p}})^{d}
\otimes_{A_{\mathfrak{p}}}B_{\mathfrak{q}}
\ar[r]^{}&
((N\otimes_{A}B)_{\mathfrak{q}})^{d}}$$
In the above diagram, the vertical maps and the bottom map are natural isomorphisms. Hence, $\phi'$ is an isomorphism. 
\end{proof} 

\section{Picard group vs the group of units of  the Grothendieck ring}

To prove the first main result of this section (Theorem \ref{coro najib 22}), we need the following key lemmas:  

\begin{lemma}\label{Lemma nabcg 20} If $e$ and $e'$ are idempotents of a ring $A$, then we have the canonical isomorphism of $A$-modules: $$Ae\simeq Ae(1-e')\oplus Ae/Ae(1-e').$$
\end{lemma}

\begin{proof} Consider the following canonical short exact sequence of $A$-modules: $$\xymatrix{0\ar[r]&Ae(1-e')\ar[r]&Ae\ar[r]&Ae/Ae(1-e')
\ar[r]&0.}$$
For each idempotent $e\in A$, we have $Ae\cap A(1-e)=0$ thus $Ae\oplus A(1-e)\simeq A$ and so $Ae$ is a (finitely generated) projective $A$-module. It follows that $Ae/Ae(1-e')$ is also a projective $A$-module, because
$Ae/Ae(1-e')\simeq Ae\otimes_{A}Ae'$. Hence, the above sequence splits. So, $Ae\simeq Ae(1-e')\oplus Ae/Ae(1-e')$.
\end{proof}

\begin{lemma}\label{lemma najib 21} If $e$ and $e'$ are orthogonal idempotents of a ring $A$, then we have the canonical isomorphism of $A$-modules $A(e+e')\simeq Ae\oplus Ae'$.
\end{lemma}

\begin{proof} The map $f:A\rightarrow Ae\oplus Ae'$ given by $f(r)=(re,re')$ is a morphism of $A$-modules. Clearly $\Ker(f)=A(1-e-e')$, because if $f(r)=0$, then $re=re'=0$ and so $r=r(1-e-e')\in A(1-e-e')$. The map $f$ is also surjective, because if $(a,b)\in A^{2}$ then $f(ae+be')=(ae,be')$. Thus $f$ induces an isomorphism of $A$-modules $A/A(1-e-e')\simeq Ae\oplus Ae'$.
We also have $\Ann(e+e')=A(1-e-e')$, because $e+e'$ is  idempotent. Thus $A(e+e')\simeq A/\Ann(e+e')\simeq Ae\oplus Ae'$.
\end{proof}

\begin{theorem}\label{coro najib 22} If $e$ and $e'$ are idempotents of a ring $A$, then we have the canonical isomorphism of $A$-modules: $$Ae\oplus Ae'\simeq Ae/Ae(1-e')\oplus Ae'/Ae'(1-e)\oplus A(e+e'-2ee').$$
\end{theorem}

\begin{proof} We may write $e+e'-2ee'=e(1-e')+e'(1-e)$. Then setting $a:=e(1-e')$ and $b:=e'(1-e)$. Clearly $ab=0$. Then by Lemma \ref{lemma najib 21},  $A(a+b)\simeq Aa\oplus Ab$. Then by applying Lemma \ref{Lemma nabcg 20}, the assertion is easily deduced.
\end{proof} 

For any ring $A$, by $\mathscr{B}(A)=\{e\in A: e=e^{2}\}$ we mean the set of all idempotent elements of $A$ which is an Abelian group under the operation $e\oplus e':=e+e'-2ee'$ (the calligraphy letter $\mathscr{B}$ stands for George Boole). For more information on this group we refer the interested reader to \cite{Tarizadeh-Taheri}.

\begin{lemma}\label{Lemma idempots iso gro} For any ring $A$, we have the canonical isomorphism of groups $\mathscr{B}(A)\simeq H_{0}(A)^{\ast}$.
\end{lemma}

\begin{proof} If $e\in A$ is an idempotent then we have a continuous map $\phi_{e}:\Spec(A)\rightarrow \mathbb{Z}$ which is defined as $\phi_{e}(\mathfrak{p})=1$ if $e\in\mathfrak{p}$ otherwise $\phi_{e}(\mathfrak{p})=-1$. We show that the map $e\mapsto\phi_{e}$ is an isomorphism of groups from the additive group
$\mathscr{B}(A)$ onto $H_{0}(A)^{\ast}$. Clearly $\phi_{e}\in H_{0}(A)^{\ast}$, because $\phi_{e}^{2}=1$ (also note that $\phi_{e}=-\phi_{1-e}$).
Then we show that the above map is a morphism of groups, i.e., $\phi_{e\oplus e'}=\phi_{e}\cdot\phi_{e'}$ for any idempotents $e$ and $e'$ of $A$. Let $\mathfrak{p}$ be a prime ideal of $A$. If both $e,e'\in\mathfrak{p}$ then $e\oplus e'\in\mathfrak{p}$ and so $\phi_{e\oplus e'}(\mathfrak{p})=1=(\phi_{e}\cdot\phi_{e'})(\mathfrak{p})$. Suppose $e\in\mathfrak{p}$ but $e'\notin\mathfrak{p}$ then $e\oplus e'\notin\mathfrak{p}$, because $e'(e\oplus e')=e'(1-e)\notin\mathfrak{p}$, so in this case $\phi_{e\oplus e'}(\mathfrak{p})=-1=(\phi_{e}\cdot\phi_{e'})(\mathfrak{p})$. Finally, suppose $e,e'\notin\mathfrak{p}$ then $e\oplus e'\in\mathfrak{p}$, because we have $e\oplus e'=e(1-e')+e'(1-e)\in\mathfrak{p}$, thus in this case $(\phi_{e}\cdot\phi_{e'})(\mathfrak{p})=
\phi_{e}(\mathfrak{p})\cdot\phi_{e'}(\mathfrak{p})=
(-1)\cdot(-1)=1=\phi_{e\oplus e'}(\mathfrak{p})$. Hence, the above map is a morphism of groups. If $\phi_{e}=\phi_{e'}$ for some idempotents $e,e'\in A$, then clearly $D(e)=D(e')$ and so $e=e'$. If $f\in H_{0}(A)^{\ast}$ then by \cite[Theorem 1.1]{A. Tarizadeh Racsam2}, there exists a (unique) idempotent $e\in A$ such that $f^{-1}(\{1\})=V(e)$. Since $f^{2}=1$, so $\phi_{e}=f$. This completes the proof.
\end{proof}

It is well known that the Picard group $\Pic(A)$ can be canonically embedded in the group of units of the Grothendieck ring $K_{0}(A)$. In the following result, not only it is proved by a new method, we also complete this observation by further involving the additive group of idempotents $\mathscr{B}(A)$. This result, in particular, paves the way to understand the structure of the group $K_{0}(A)^{\ast}$ for a certain class of rings (see Corollary \ref{coro lbp 24}).

\begin{theorem}\label{Theorem Pic-iv} For any ring $A$ we have the following complex of Abelian groups which is exact at  $\Pic(A)$ and $\mathscr{B}(A)$: $$\xymatrix{0\ar[r]&\Pic(A)\ar[r]^{f}&K_{0}(A)^{\ast}
\ar[r]^{g}&\mathscr{B}(A)\ar[r]&0.}$$
In addition, there exists a morphism of groups $h:\mathscr{B}(A)\rightarrow K_{0}(A)^{\ast}$ such that $gh$ is the identity map.
\end{theorem}

\begin{proof} We first show that the map $f:\Pic(A)\rightarrow K_{0}(A)^{\ast}$ given by $L\mapsto[L,0]$ is an injective morphism of groups. If $L\in\Pic(A)$ then $L\otimes_{A}\widehat{L}\simeq A$. Thus $[L,0]\cdot[\widehat{L},0]=[A,0]$. Hence, $[L,0]$ is invertible in $K_{0}(A)$ and so the above map is well-defined. This map is clearly a morphism of groups. For injectivity, suppose $[L,0]=[A,0]$. Then there exists some natural number $n\geqslant0$ such that $L\oplus A^{n}\simeq A^{n+1}$ as $A$-modules. It suffices to show that $L\simeq A$ as $A$-modules. We will use exterior powers to obtain such an isomorphism. It is well known that for any two modules $M$ and $N$ over a ring $A$, we have the canonical isomorphism of $A$-modules $\bigwedge^{k}(M\oplus N)\simeq\bigoplus\limits_{p+q=k}\bigwedge^{p}(M)
\otimes_{A}\bigwedge^{q}(N)$. It is also well known that if $F$ is a free $A$-module of rank $d\geqslant0$, then $\Lambda^{k}(F)$ is a free $A$-module of rank $\binom{d}{k}$
for all $0\leqslant k\leqslant d$ and $\Lambda^{k}(F)=0$ for all $k>d$. Since $L$ is a finitely generated projective $A$-module of rank 1, thus $\bigwedge^{k}(L)=0$ for all $k\geqslant2$, because if $\mathfrak{p}\in\Spec(A)$ then we have the canonical isomorphisms of $A_{\mathfrak{p}}$-modules $\big(\bigwedge_{A}^{k}(L)\big)_{\mathfrak{p}}\simeq
\bigwedge_{A}^{k}(L)\otimes_{A}A_{\mathfrak{p}}\simeq
\bigwedge_{A_{\mathfrak{p}}}^{k}(L_{\mathfrak{p}})\simeq
\bigwedge_{A_{\mathfrak{p}}}^{k}(A_{\mathfrak{p}})=0$ for all $k\geqslant2$.
Now using these observations, we have $A\simeq\bigwedge^{n+1}(A^{n+1})
\simeq\bigwedge^{n+1}(L\oplus A^{n})\simeq\bigoplus\limits_{p+q=n+1}\bigwedge^{p}(L)
\otimes_{A}\bigwedge^{q}(A^{n})\simeq\bigwedge^{1}(L)
\otimes_{A}\bigwedge^{n}(A^{n})\simeq L\otimes_{A}A\simeq L$ as $A$-modules. Next, we construct the map $g$.  
The canonical ring map $K_{0}(A)\rightarrow H_{0}(A)$ given by $[M,N]\mapsto\mathrm{r}_{M}-\mathrm{r}_{N}$ induces a group map $K_{0}(A)^{\ast}\rightarrow H_{0}(A)^{\ast}$.
Then using this and Lemma \ref{Lemma idempots iso gro}, we obtain a group map $g:K_{0}(A)^{\ast}\rightarrow\mathscr{B}(A)$ given by $[M,N]\mapsto e$ where $e\in A$ is a (unique) idempotent with $(\mathrm{r}_{M}-\mathrm{r}_{N})^{-1}(\{1\})=V(e)$. If $e\in A$ is an idempotent, then $1-2e$ is invertible in $A$, because $(1-2e)^{2}=1$. Also, the element $[Ae,0]$ is an idempotent of $K_{0}(A)$, because
$Ae\otimes_{A}Ae\simeq Ae\otimes_{A}A/(1-e)\simeq Ae/Ae(1-e)\simeq Ae$
and so $[Ae,0]\cdot[Ae,0]=[Ae\otimes_{A}Ae,0]=[Ae,0]$.
Using these observations, we obtain a function $h:\mathscr{B}(A)\mapsto K_{0}(A)^{\ast}$ which is defined by $e\mapsto[A,Ae\oplus Ae]$. By Theorem \ref{coro najib 22}, $h$ is a group morphism.
It can be easily seen that $gh$ is the identity map. Finally, we show that $gf=0$, i.e., $\Ima(f)\subseteq\Ker(g)$.
If $L\in\Pic(A)$ then $L$ is a finitely generated projective $A$-module of rank 1. Hence, its rank map is the constant function at 1. Thus we have $(\mathrm{r}_{L}-\mathrm{r}_{0})^{-1}(\{1\})=
(\mathrm{r}_{L})^{-1}(\{1\})=\Spec(A)=V(0)$. This shows that $g([L,0])=0$.
\end{proof}

\begin{remark} In the proof of Theorem \ref{Theorem Pic-iv}, we observed that if $M$ is a finitely generated projective $A$-module of rank 1, then from $M\oplus A^{n}\simeq A^{n+1}$ we obtained that $M\simeq A$. But it is important to notice that this conclusion does not hold in general. More precisely, let $M$ be a module over a ring $A$ such that there exists natural numbers $d,n\geqslant0$ for which $M\oplus A^{d}\simeq A^{n}$ as $A$-modules (in this case, $M$ is called a stably free module). Then clearly $d\leqslant n$ and $M$ is a finitely generated projective $A$-module of constant rank $n-d$. If $n-d\geqslant2$ then $M$ is not necessarily a free module. That is, there are stably free modules which are not free (see e.g. \cite[Chap. 3]{Ischebeck-Rao}, \cite[p. 301]{Lam} or \cite[Chap. XXI, \S2]{Lang}).
\end{remark}

The complex of Theorem \ref{Theorem Pic-iv} is not exact in general. In fact, Pierre Deligne pointed out to us that by using Jouanolou’s trick, the exactness of this sequence fails for arbitrary rings. However, it is natural to ask when the complex of Theorem \ref{Theorem Pic-iv} will be exact. This leads us to the following notion: we say that a ring $A$ has the \emph{line bundle property} if whenever $M$ is a finitely generated projective $A$-module of constant rank $d+1$ with $d\geqslant0$, then $A^{n}\oplus M\simeq A^{n+d}\oplus\Lambda^{d+1}(M)$ as $A$-modules for some $n\geqslant0$. It is well known (due to J.P. Serre) that every Noetherian one dimensional ring has the line bundle property (with $n=0$). Its proof can be found in \cite[Chap I, Proposition 3.4]{Weibel}. The ``line bundle property'' is a typically 1-dimensional phenomenon.

\begin{corollary}\label{coro lbp 24} If a ring $A$ has the line bundle property, then we have the following split exact sequence of Abelian groups: $$\xymatrix{0\ar[r]&\Pic(A)\ar[r]&K_{0}(A)^{\ast}
\ar[r]&\mathscr{B}(A)\ar[r]&0.}$$
\end{corollary}

\begin{proof} By Theorem \ref{Theorem Pic-iv}, it suffices to show that $\Ker(g)\subseteq\Ima(f)$ where $f:\Pic(A)\rightarrow K_{0}(A)^{\ast}$ and $g:K_{0}(A)^{\ast}\rightarrow\mathscr{B}(A)$.
Take $[M,A^{d}]\in\Ker(g)$ where $M$ is a finitely generated projective $A$-module and $d\geqslant0$. It follows that $(\mathrm{r}_{M}-\mathrm{r}_{A^{d}})^{-1}(\{1\})=
(\mathrm{r}_{M}-d)^{-1}(\{1\})=V(0)=\Spec(A)$. This shows that $M$ is of constant rank $d+1$. Thus $L:=\Lambda^{d+1}(M)$ is a finitely generated projective $A$-module of rank $1$ and so $L\in\Pic(A)$. Since $A$ has the line bundle property, $A^{n}\oplus M\simeq A^{n+d}\oplus L$ as $A$-modules for some $n\geqslant0$. This shows that $[M,A^{d}]=[L,0]\in\Ima(f)$.
\end{proof}

In particular, if a ring $A$ has the line bundle property with no nontrivial idempotents, then $K_{0}(A)^{\ast}\simeq\Pic(A)\oplus\mathbb{Z}/2$ where $\mathbb{Z}/2=\{0,1\}$ is the additive group of integers modulo 2.

\begin{corollary} If a ring $A$ has the line bundle property, then the following assertions hold: \\
$\mathbf{(i)}$ If $\Min(A)$ is a finite set, then $K_{0}(A)^{\ast}\simeq\Pic(A)\oplus(\mathbb{Z}/2)^{d}$. \\
$\mathbf{(ii)}$ If $\Max(A)$ is a finite set, then $K_{0}(A)^{\ast}\simeq(\mathbb{Z}/2)^{d}$. \\
In addition, $d\geqslant0$ is the number of connected components of $\Spec(A)$. 
\end{corollary}

\begin{proof} By Corollary \ref{coro lbp 24}, we have $K_{0}(A)^{\ast}\simeq\Pic(A)\oplus\mathscr{B}(A)$.
If a ring has finitely many minimal primes or finitely many maximal ideals, then it has finitely many idempotents and so $\Spec(A)$ has finitely many connected components. Then using the Chinese Remainder Theorem and \cite[Lemma 4.7]{A. Tarizadeh Racsam2}, we observe that the group $\mathscr{B}(A)$ is isomorphic to the additive group $(\mathbb{Z}/2)^{d}$. Note that if $A$ has finitely many maximal ideals, then its Picard group is trivial. 
\end{proof}

Every Noetherian one dimensional ring satisfies in the hypothesis of the above result. In particular, if $A$ is a Dedekind domain then we have the canonical isomorphism of groups  $K_{0}(A)^{\ast}\simeq\Cl(A)\oplus\mathbb{Z}/2$.

\begin{corollary} Every Abelian group $G$ can be embedded in $K_{0}(A)^{\ast}$ for some Dedekind domain $A$ and $K_{0}(A)^{\ast}\simeq G\oplus\mathbb{Z}/2$.
\end{corollary}

\begin{proof} It is well known that every Abelian group $G$ is isomorphic to the ideal class group of a  Dedekind domain $A$  (see \cite[Theorem 7]{Claborn}). The ideal class group of every integral domain, and more generally of any reduced ring with finitely many minimal primes, is canonically isomorphic to its Picard group.  Every Dedekind domain has the line bundle property. Using these facts, then the assertion follows from Corollary \ref{coro lbp 24}.
\end{proof}

Let $A$ be a ring. It is well known that the canonical rank map $K_{0}(A)\rightarrow H_{0}(A)$ given by $[M,N]\mapsto\mathrm{r}_{M}-\mathrm{r}_{N}$ is a surjective morphism of rings and its kernel is precisely the nil-radical of $A$. Showing that the kernel of this map is contained in the nil-radical is the most difficult and technical part of the proof (for details see \cite[Chap. IX, \S3, Proposition 4.6]{Bass} or  \cite[Corollary 10.7]{Swan} or \cite[Chap. II, \S4, Corollary 4.6.1]{Weibel}). Then we have the canonical isomorphism of rings $K_{0}(A)_{\mathrm{red}}\simeq H_{0}(A)$. This fundamental identification leads us to the following results.

\begin{theorem}\label{Corollary idemps of gr} For any ring $A$, we have the canonical isomorphism of groups $\mathscr{B}(A)\simeq\mathscr{B}\big(K_{0}(A)\big)$.
\end{theorem}

\begin{proof} We show that the map $\mathscr{B}(A)\rightarrow\mathscr{B}\big(K_{0}(A)\big)$ given by $e\mapsto[Ae,0]$ is an isomorphism of groups. In the proof of Theorem \ref{Theorem Pic-iv}, we observed that $[Ae,0]$ is idempotent.
By Theorem \ref{coro najib 22}, this map is a group morphism. Suppose $[Ae,0]=[Ae',0]$ for some idempotents $e,e'\in A$. To prove $e=e'$ it suffices to show that $D(e)=D(e')$. There exists some $n\geqslant0$ such that $Ae\oplus A^{n}\simeq Ae'\oplus A^{n}$ as $A$-modules. Thus the rank maps of $Ae\oplus A^{n}$ and $Ae'\oplus A^{n}$ are the same, and so $\mathrm{r}_{Ae}=\mathrm{r}_{Ae'}$. Now if $\mathfrak{p}\in D(e)$ then $1=\mathrm{r}_{Ae}(\mathfrak{p})=
\mathrm{r}_{Ae'}(\mathfrak{p})$. This shows that $\mathfrak{p}\in D(e')$. Similarly, $D(e')\subseteq D(e)$.
Finally, we show that the above map is surjective. If $z\in K_{0}(A)$ is an idempotent then its image $g:=\phi(z)$ under the canonical ring map $\phi:K_{0}(A)\rightarrow H_{0}(A)$ is idempotent. It follows that $g(\mathfrak{p})\in\{0,1\}$ for all $\mathfrak{p}\in\Spec(A)$.
Since $g:\Spec(A)\rightarrow\mathbb{Z}$ is a continuous map, there exists an idempotent $e\in A$ such that $g^{-1}(\{1\})=D(e)$. This shows that
$\mathrm{r}_{Ae}=g$. Thus $[Ae,0]-z$ is contained in $\Ker(\phi)$. But $\Ker(\phi)$ is the nil-radical of $K_{0}(A)$. It can be easily seen that if $e$ and $e'$ are idempotents of a ring $A$ such that $e-e'$ is contained in the Jacobson radical of $A$, then $e=e'$. Therefore  $z=[Ae,0]$.
\end{proof}

We say that a morphism of rings $f:A\rightarrow B$ lifts idempotents if $e'\in B$ is an idempotent then there exists an idempotent $e\in A$ such that $f(e)=e'$. 

\begin{corollary}\label{coro 21 lift} A morphism of rings $f:A\rightarrow B$ lifts idempotents if and only if $K_{0}(f):K_{0}(A)\rightarrow K_{0}(B)$ lifts idempotents.
\end{corollary}

\begin{proof} Assume $f$ lifts idempotents. By Theorem \ref{Corollary idemps of gr}, each idempotent of $K_{0}(B)$ is of the form $[Be',0]$ where $e'\in B$ is  idempotent. So there exists an idempotent $e\in A$ such that $f(e)=e'$. This yields that $Ae\otimes_{A}B\simeq A/A(1-e)\otimes_{A}B\simeq B/B(1-e')\simeq Be'$. This shows that the image of the idempotent $[Ae,0]$ under the map $K_{0}(f)$ equals $[Be',0]$. Hence, $K_{0}(f)$ lifts idempotents. Conversely, if $e'\in B$ is an idempotent then by hypothesis, there exists an idempotent $e\in A$ such that $[Bf(e),0]=[Be',0]$. But in the proof of Theorem \ref{Corollary idemps of gr} we observe that in this case, $f(e)=e'$. Hence, $f$ lifts idempotents.
\end{proof}

\begin{theorem}\label{coro 23 surj} If a morphism of rings $f:A\rightarrow B$ lifts idempotents and $B$ has finitely many maximal ideals, then $K_{0}(f):K_{0}(A)\rightarrow K_{0}(B)$ is surjective.
\end{theorem}

\begin{proof} Take $z'\in K_{0}(B)$. Then consider the following diagram of commutative rings: $$\xymatrix{
K_{0}(A)\ar[r]^{K_{0}(f)}\ar[d]^{\phi} &K_{0}(B)\ar[d]^{\psi}\\H_{0}(A)
\ar[r]^{H_{0}(f)}&H_{0}(B)}$$ where the vertical arrows are the canonical rank maps. Since $f$ lifts idempotents, then by \cite[Theorem 5.2]{A. Tarizadeh Racsam2}, $H_{0}(f)$ is surjective. The canonical map $\phi$ is also surjective. Thus there exists some $z\in K_{0}(A)$ such that $\psi(z')=H_{0}(f)\phi(z)$.  We know that if $M$ is a finitely generated  projective (resp. flat) $A$-module, then $M\otimes_{A}B$ is a finitely generated  projective (resp. flat) $B$-module and we have $\mathrm{r}_{M\otimes_{A}B}=\mathrm{r}_{M}\circ f^{\ast}$. This shows that the above diagram is commutative. It follows that
$z'-z''\in\Ker(\psi)$ where $z''$ is the image of $z$ under $K_{0}(f)$. We know that $\Ker(\psi)$ is the nil-radical of $K_{0}(B)$. But the nil-radical of $K_{0}(B)$ is precisely the set of all elements of the form $[N,B^{d}]$ where $N$ is a finitely generated projective $B$-module of rank $d\geqslant0$. It is well known that every finitely generated projective (even flat) module of constant rank over a ring with finitely many maximal ideals is a free module (see \cite[Tags 00NX, 00NZ, 02M9]{Johan}). Therefore $N\simeq B^{d}$ as $B$-modules and so $[N,B^{d}]=0$. This shows that  $K_{0}(B)$ is a reduced ring. Thus $z'=z''=K_{0}(f)(z)$. Hence, $K_{0}(f)$ is surjective.
\end{proof}

Let $A$ be a local ring or a $\mathrm{PID}$. It is well known that every projective module over $A$ is a free $A$-module. In particular, the semiring of the isomorphism classes of finitely generated projective $A$-modules is isomorphic to the semiring of the natural numbers $\mathbb{N}$. Hence, the ring $K_{0}(A)$ is isomorphic to the ring of integers $\mathbb{Z}=K_{0}(\mathbb{N})$. 

Using the above observation, if $\mathfrak{p}$ is a prime ideal of a ring $A$ then the canonical ring map $\pi:A\rightarrow A_{\mathfrak{p}}$ induces a  ring map $K_{0}(\pi):K_{0}(A)\rightarrow K_{0}(A_{\mathfrak{p}})=\mathbb{Z}$. The ring map $K_{0}(\pi)$ is surjective, because its composition with the natural ring map $\mathbb{Z}\rightarrow K_{0}(A)$ is the identity map of $\mathbb{Z}$. In particular, if $A$ is a nonzero ring, then the ring $K_{0}(A)$ is infinite.

Similarly, if $\mathfrak{m}$ is a maximal ideal of $A$ then the canonical ring map $\pi:A\rightarrow A/{\mathfrak{m}}$ induces a surjective ring map $K_{0}(\pi):K_{0}(A)\rightarrow K_{0}(A/{\mathfrak{m}})=\mathbb{Z}$.

\begin{corollary}\label{coro 22 cz} If $A$ is a nonzero ring, then $K_{0}(A)$ and $H_{0}(A)$ are infinite rings of characteristic zero.
\end{corollary}

\begin{proof} In the above argument we observed that $K_{0}(A)$ is an extension ring of $\mathbb{Z}$. Hence, $K_{0}(A)$ is of characteristic zero. In addition, there is a prime ideal $P$ in $K_{0}(A)$ such that the ring $K_{0}(A)/P$ is isomorphic to $\mathbb{Z}$. Then we obtain a surjective (ring) map $H_{0}(A)\simeq K_{0}(A)_{\mathrm{red}}\rightarrow K_{0}(A)/P\simeq\mathbb{Z}$. Thus $H_{0}(A)$ is also infinite. In fact, $H_{0}(A)$ is an extension ring of $\mathbb{Z}$, and so it is of characteristic zero.
\end{proof}

The following result is indeed a reformulation of a well-known fact in the literature:  

\begin{theorem} If $I$ is an ideal of a ring $A$ such that $A$ is complete with respect to the $I$-adic topology, then we have the following exact sequence of Abelian groups: $$\xymatrix{0\ar[r]&1+I\ar[r]& A^{\ast}\ar[r]&(A/I)^{\ast}
\ar[r]&0.}$$
\end{theorem}

\begin{proof} By \cite[Lemma 2.1]{Tarizadeh new}, it will be enough to show that $I$ is contained in the Jacobson radical of $A$. To see this, it suffices to show that $1-a$ is invertible in $A$ for all $a\in I$. Clearly the sequence $(\sum\limits_{k=0}^{n-1}a^{k}+I^{n})_{n\geqslant1}=
(1+I,1+a+I^{2},1+a+a^{2}+I^{3},\cdots)$ is an element of the ring $\lim\limits_{\overleftarrow{n\geqslant1}}A/I^{n}$.
By hypothesis, the canonical ring map $\pi:A\rightarrow\lim\limits_{\overleftarrow{n\geqslant1}}
A/I^{n}$ given by $r\mapsto(r+I^{n})_{n\geqslant1}$ is surjective. So there exists some $b\in A$ such that $\sum\limits_{k=0}^{n-1}a^{k}-b\in I^{n}$ for all $n\geqslant1$. Then we show that $(1-a)b=1$. To see this, it will be enough to show that $(1-a)b-1\in\Ker(\pi)=0$. But we have $b=\sum\limits_{k=0}^{n-1}a^{k}-c_{n}$ for some $c_{n}\in I^{n}$ and so
$(1-a)b-1=(a-1)c_{n}-a^{n}\in I^{n}$ for all $n\geqslant1$. This completes the proof. 
\end{proof}

For any ring $A$, the subgroups $1+\mathfrak{N}$ and $1+\mathfrak{J}$ of $A^{\ast}$ are respectively called the group of unipotents of $A$ and the group of uni-Jacobsons of $A$ where $\mathfrak{N}$ and $\mathfrak{J}$ are the nil-radical and Jacobson radical of $A$.

In the following result, we characterize finitely generated projective modules in terms of orthogonal idempotents:   

\begin{theorem}\label{Theorem new charct proj mod} Let $M$ be a finitely generated module over a ring $A$. Then $M$ is a projective $A$-module if and only if there exists a finite sequence $e_{0},\ldots,e_{n}$ of orthogonal idempotents of $A$ such that $\sum\limits_{k=0}^{n}e_{k}=1$ and $M_{\mathfrak{p}}\simeq (A_{\mathfrak{p}})^{{k}}$ as $A_{\mathfrak{p}}$-modules
for all $\mathfrak{p}\in D(e_{k})$. In this case, the annihilator of the $A$-module $\Lambda^{k}(M)$ is generated by the idempotent $\sum\limits_{i=0}^{k-1}e_{i}$ for all $k\in\{1,\ldots,n, n+1\}$.
\end{theorem}

\begin{proof} If $M$ is $A$-projective, then its rank map $\mathrm{r}_{M}:\Spec(A)\rightarrow\mathbb{Z}$ is continuous. Using the quasi-compactness of the prime spectrum, there exists a natural number $n\geqslant0$ such that $\Spec(A)=\bigcup\limits_{k=0}^{n}\mathrm{r}_{M}^{-1}
(\{k\})$. Clearly each $\mathrm{r}_{M}^{-1}(\{k\})$ is a clopen (both open and closed) subset of $\Spec(A)$. Thus there exists an idempotent $e_{k}\in A$ such that $\mathrm{r}_{M}^{-1}(\{k\})=D(e_{k})$. Now the desired assertion is easily deduced. Next, we prove the reverse implication. By the hypothesis, $M$ is a flat $A$-module, because flatness is a local property. Again by the hypothesis, the rank map of $M$ is continuous. Hence, $M$ is $A$-projective. Now we show that the annihilator of $\Lambda^{k}(M)$ is generated by $\sum\limits_{i=0}^{k-1}e_{i}$. Indeed, we have the canonical isomorphism of $A_{\mathfrak{p}}$-modules $\Lambda_{A}^{k}(M)\otimes_{A}A_{\mathfrak{p}}\simeq
\Lambda_{A_{\mathfrak{p}}}^{k}(M_{\mathfrak{p}})$. Also remember that if $F$ is a free $A$-module of rank $d\geqslant0$, then $\Lambda^{k}(F)$ is a free $A$-module of rank $\binom{d}{k}$ and hence
$\Lambda^{k}(F)\neq0$ for all $0\leqslant k\leqslant d$ and $\Lambda^{k}(F)=0$ for all $k>d$. Therefore $\Supp\big(\Lambda^{k}(M)\big)=\bigcup\limits_{i=k}^{n}D(e_{i})
=D(\sum\limits_{i=k}^{n}e_{i})$. Since $\Lambda^{k}(M)$ is a finitely generated projective $A$-module, so its annihilator is generated by an idempotent element $e\in A$, because the annihilator of every finitely generated projective module is generated by an idempotent element (see e.g. \cite[Corollary 3.2]{A. Tarizadeh acta}). Thus $\Supp\big(\Lambda_{A}^{k}(M)\big)=V(e)=D(1-e)$.
It follows that $1-e=\sum\limits_{i=k}^{n}e_{i}$, hence  $e=\sum\limits_{i=0}^{k-1}e_{i}$. 
\end{proof}

\begin{example} We illustrate Theorem \ref{Theorem new charct proj mod} with two examples. If $A$ is a ring then for $M=A^{2}$ we have the sequence $e_{0}=e_{1}=0$ and $e_{2}=1$. As another example, if $e\in A$ is an idempotent then for the projective $A$-module $M=Ae$ we have the sequence $e_{0}=1-e$ and $e_{1}=e$.
\end{example}


\end{document}